\documentclass[11pt,a4paper,fleqn]{scrartcl}
\usepackage{geometry}
\usepackage[applemac]{inputenc}
\usepackage[italian,english]{babel}
\usepackage{amsmath, amsfonts, amsthm,xcolor,amssymb}
\usepackage{paralist}
\numberwithin{equation}{section}
\usepackage{mathtools, mathabx}
\usepackage[colorlinks=true,linkcolor=blue,citecolor=blue]{hyperref}
\usepackage{xcolor}
\newtheorem{thm}{Theorem}[section]

\newtheorem{prop}[thm]{Proposition}

\theoremstyle{definition}
\newtheorem{rem}[thm]{Remark}
\theoremstyle{remark}

\newcommand{\ds}{\displaystyle}

\newcommand{\R}{\mathbb{R}}

\newcommand{\de}{\partial}
\newcommand{\eps}{\varepsilon}

\newcommand\restr[2]{{% we make the whole thing an ordinary symbol
  \left.\kern-\nulldelimiterspace % automatically resize the bar with \right
  #1 % the function
  \vphantom{ |} % pretend it's a little taller at normal size
  \right|_{#2} % this is the delimiter
  }}
{\left\{\begin{array}{@{}l@{}}}{\end{array}\right.}
\patchcmd{\abstract}{\scshape\abstractname}{\textbf{\abstractname}}{}{}
\makeatletter %note a di pagina senza numero 1
\def\@makefnmark{} %note a di pagina senza numero 2
\makeatother %note a di pagina senza numero 3

\begin{document}

\title{An optimal bound for nonlinear eigenvalues and torsional rigidity on domains with  holes}
 \author{Francesco Della Pietra$^{*}$%}
       \thanks{Dipartimento di Matematica e Applicazioni ``R. Caccioppoli'', Universit\`a degli studi di Napoli Federico II, Via Cintia, Monte S. Angelo - 80126 Napoli, Italia.  \newline 
       Email: f.dellapietra@unina.it}
       \\ Gianpaolo Piscitelli$^{\dag}$\thanks{Dipartimento di Ingegneria Elettrica e dell'Informazione  ``M. Scarano'', Universit\`a degli Studi di Cassino e del Lazio Meridionale Via G. Di Biasio n. 43, 03043 Cassino (FR), Italy. \newline Email: gianpaolo.piscitelli@unicas.it}
}
%}
%\title{An optimal upper bound for nonlinear eigenvalues on domains with convex holes}
%
% \author{Francesco Della Pietra$^{*}$, \\ Gianpaolo Piscitelli$^{**}$%
%       \thanks{Universit\`a degli studi di Napoli Federico II, Dipartimento di Matematica e Applicazioni ``R. Caccioppoli'', Via Cintia, Monte S. Angelo - 80126 Napoli, Italia. \\
%       Dipartimento di Ingegneria Elettrica e dell'Informazione  ``M. Scarano'', Universit\`a degli Studi di Cassino e del Lazio Meridionale\\ Via G. Di Biasio n. 43, 03043 Cassino (FR), Italy. 
%       Email: f.dellapietra@unina.it, gianpaolo.piscitelli@unicas.it}
%}
\date{}
\maketitle

\begin{abstract}
\noindent{\textsc{Abstract.}} In this paper we prove an optimal upper bound for the first eigenvalue of a Robin-Neumann boundary value problem for the $p$-Laplacian operator in domains with convex holes. An analogous estimate is obtained for the corresponding torsional rigidity problem.
\end{abstract}

\section{Introduction}

Let $\beta>0$, and $\Omega$ and $D$ be two bounded open sets, with $D$ convex, $\Omega$ Lipschitz, connected, and $D\Subset \Omega$. We will denote by $\Sigma=\Omega\setminus \bar D$.

Let us consider the following problem:
\begin{equation}
\label{minpb}
\lambda(\beta,\Sigma)=  \min\left\{ \dfrac{\ds\int _{\Sigma}|\nabla \varphi|^p\;dx+\beta\int_{\de D}  |\varphi |^{p}d\sigma}{\ds\int_{\Sigma}|\varphi|^p\;dx},\; \varphi\in W^{1,p}(\Sigma),\;\varphi\not\equiv 0\right\},
\end{equation}
where $\beta$ is a positive constant and $1<p<+\infty$. A minimizer  $u\in W^{1,p}(\Sigma)$ of \eqref{minpb} satisfies 
\begin{equation}\label{case1intro}
\begin{cases}
	-\Delta_p u=\lambda(\beta,\Sigma) |u|^{p-2} u & \mbox{in}\ \Sigma \vspace{0.2cm}\\
|\nabla u|^{p-2}\dfrac{\de u}{\de \nu}=0 &\mbox{on}\ \de\Omega \vspace{0.2cm}\\ 
| \nabla u|^{p-2}\dfrac{\de u}{\de \nu}+\beta |u|^{p-2}u=0 &\text{on }\de D.
\end{cases}
\end{equation}
In this paper we obtain an optimal upper bound on $\lambda(\beta,\Sigma)$, when $\beta>0$ is fixed and $\Sigma=\Omega\setminus \bar D$ %, with $D$ convex and $\Omega\supset D$ Lipschitz, 
varies among all domains such that the volume of $\Sigma$ and the $(n-1)$-quermassintegral of $D$ are given (see Section 2.3 for the precise definition). If $n=2$, the geometrical constraint on $D$ corresponds to fix its perimeter. In particular, we show that $\lambda(\beta,\Sigma)$ is maximized by the spherical shell. The first main result is the following.

\begin{thm}
\label{herschn}
	Let $\beta>0$, and $\Omega$ and $D$ be two bounded open sets, with $D$ convex, $\Omega$ Lipschitz, connected, and $D\Subset \Omega$. Let be $\Sigma=\Omega\setminus \bar D$, $A=A_{R_1,R_2}=B_{R_{2}}\setminus \overline{B}_{R_{1}}$, where $B_{R_{i}}$ is a ball centered at the origin with radius $R_{i}$, $i=1,2$. Suppose that $\left|A\right|=\left|\Sigma\right|$, and $W_{n-1}(B_{R_1})=W_{n-1}(D)$. Then,
\begin{equation*}
	\label{in_eig_NR}
	\lambda(\beta,\Sigma)\le \lambda(\beta,A).
			% \label{in_tor_NR}			T_p^{NR}(\beta, \Omega)\geq T_p^{NR}(\beta, A).%\\	%\label{in_eig_NR}		\lambda_p^{NR}(\beta, \Omega)\leq \lambda_p^{NR}(\beta, A),\\
		%	 \label{in_tor_NR}				T_p^{NR}(\beta, \Omega)\geq T_p^{NR}(\beta, A).
\end{equation*}
%the annulus $A$ maximizes the first $p$-Laplacian eigenvalue and minimizes the torsion rigidity.
\end{thm}

%\begin{enumerate}
%\item[\textbf{Main Theorem.}] Let $\beta>0$ and $\Sigma=\Omega\setminus \bar D$ as before, then 
%\[
%\lambda(\beta,\Sigma)\leq\lambda(\beta, B_{R_{2}}\setminus \bar B_{R_{1}}),
%\]
%where $B_{R_{2}}\setminus \bar B_{R_{1}}$ is the spherical shell with $|B_{R_{2}}\setminus \bar B_{R_{1}}|=|\Sigma|$ and $W_{n-1}(B_{R_{1}})=W_{n-1}(D)$.
%\end{enumerate}
In the case $p=n=2$, Theorem \ref{herschn} recovers a result proved by Hersch in \cite{He}. Our result generalizes it to the $p$-Laplacian operator and to any dimension, giving an answer to a question posed by Henrot in \cite[Chapter 3, Open problem 5]{H}. Actually, we are not able to prove or disprove that the first eigenvalue $\lambda(\beta,\Sigma)$ is maximized on the spherical shell when the perimeter of $D$, and not the $W_{n-1}$, is fixed. We stress that a related result has been recently proved for an optimal insulating problem (see \cite{dnt}).

Optimal estimates for eigenvalues with different Robin or Robin-Neumann boundary conditions have been proved by several authors. For example, if $D=\emptyset$, and a Robin condition on $\de\Omega$ 
is given, with $\beta>0$, a Faber-Krahn inequality has been proved in \cite{B,D} for $p=2$ and in \cite{BD,DF} for $p\in]1,+\infty[$, stating that the first eigenvalue is minimum on the ball of the same volume of $\Omega$. Otherwise, if $\beta$ is negative the problem is still open; in this direction, in \cite{FK} the authors showed that, among planar domains of the same measure, the disk is a maximizer only for small value of the parameter. On the other hand, if the perimeter rather than the volume is fixed, the ball maximizes the first eigenvalue among all open, bounded, convex, smooth enough sets  (see \cite{AFK, BFNT}). %Furthermore, we remark that in the case of a general Finsler metric, similar results holds for the anisotropic $p$-Laplacian with Dirichlet (\cite{BFK,DGP1}), Neumann (\cite{DGP2,Pi}) or Robin (\cite{GT,PT}) boundary conditions.
%In the case of $D=\emptyset$ and Neumann boundary conditions on $\de\Omega$, the celebrated Szego

In this framework, the case of Robin eigenvalue problems when $\beta\ge 0$ is not a constant has been considered in \cite{gt}, and, for example, in \cite{bbn,dpnst,DGK} where optimization with respect to $\beta$ is considered. 

Coming back to the problem with mixed boundary conditions, we recall that the case of a Neumann condition on $\de D$ and Robin condition on $\de \Omega$ has been considered in \cite{PW,ppt}.
 
%Finally, if $D=\emptyset$, and a Neumann condition on $\de\Omega$ is given, sharp inequalities for different kind of operators have been proved in \cite{S,We} and recently in \cite{?}

Another problem we deal with is the $p$-torsional rigidity with the same boundary conditions, namely
\begin{equation}
\label{minpbtor}
T(\beta,\Sigma)=  \max\left\{ \dfrac{\left(\ds\int_{\Sigma}|\varphi|\;dx\right)^{p}}{\ds\int _{\Sigma}|\nabla \varphi|^p\;dx+\beta\int_{\de D}  |\varphi |^{p}d\sigma},\; \varphi\in W^{1,p}(\Sigma),\; \varphi\not\equiv 0\right\},
\end{equation}
with $\beta>0$. The unique function $u_{\Sigma}>0$ such that
\begin{equation}\label{case1introtor}
\begin{cases}
	-\Delta_p u_{\Sigma}=1 & \mbox{in}\ \Sigma \vspace{0.2cm}\\
|\nabla u_{\Sigma}|^{p-2}\dfrac{\de u_{\Sigma}}{\de \nu}=0 &\mbox{on}\ \de\Omega \vspace{0.2cm}\\ 
| \nabla u_{\Sigma}|^{p-2}\dfrac{\de u_{\Sigma}}{\de \nu}+\beta |u_{\Sigma}|^{p-2}u_{\Sigma}=0 &\text{on }\de D,
\end{cases}
\end{equation}
is a maximizer of \eqref{minpbtor}, and
\[
T(\beta,\Sigma)=\left(\int_{\Sigma}u_{\Sigma}\,dx\right)^{p-1}.
\]
Moreover, any maximizer of \eqref{minpbtor} is proportional to $u_{\Sigma}$. 

We prove that the spherical shell minimizes $T(\beta,\Sigma)$ among all domains such that the volume of $\Sigma$ and the $(n-1)$-quermassintegral of $D$ are given. 

\begin{thm}
\label{herschntor}
	Let $\beta>0$, and $\Omega$ and $D$ be two bounded open sets, with $D$ convex, $\Omega$ Lipschitz, connected, and $D\Subset \Omega$. Let be $\Sigma=\Omega\setminus \bar D$, $A=A_{R_1,R_2}=B_{R_{2}}\setminus \overline{B}_{R_{1}}$, where $B_{R_{i}}$ is a ball centered at the origin with radius $R_{i}$, $i=1,2$. Suppose that $\left|A\right|=\left|\Sigma\right|$, and $W_{n-1}(B_{R_1})=W_{n-1}(D)$. Then,
\begin{equation*}
	\label{in_eig_NR_tor}
	T(\beta,\Sigma)\ge T(\beta,A).
			% \label{in_tor_NR}			T_p^{NR}(\beta, \Omega)\geq T_p^{NR}(\beta, A).%\\	%\label{in_eig_NR}		\lambda_p^{NR}(\beta, \Omega)\leq \lambda_p^{NR}(\beta, A),\\
		%	 \label{in_tor_NR}				T_p^{NR}(\beta, \Omega)\geq T_p^{NR}(\beta, A).
\end{equation*}
%the annulus $A$ maximizes the first $p$-Laplacian eigenvalue and minimizes the torsion rigidity.
\end{thm}

Regarding optimal estimates related to the torsional rigidity when $D=\emptyset$ and a Robin boundary condition is given on $\de\Omega$, a Saint-Venant inequality can be proved: we refer the reader to \cite{ant,bg}.

Finally, we recall that in this context, estimates have been obtained also for a more general class of problems, involving the so called Finsler operator. We refer the reader, for example, to \cite{dgR,DGP2,PT}.

The structure of the paper is the following. In the second section, we prove some properties of the first eigenvalue $\lambda(\beta,\Sigma)$ and of the torsional rigidity $T(\beta,\Sigma)$, as well as we recall some basic tool of convex analysis. In the third section, we prove the main results.

\section{Notation and preliminaries}
\subsection{Eigenvalue problem}
The following result can be proved by a standard argument of Calculus of Variations.
\begin{prop}
Let $\beta>0$, $D,\Omega$ Lipschitz domains with $D\Subset \Omega$, and $\Sigma=\Omega\setminus \bar D$. 
\begin{itemize}
\item There exists a positive minimizer $u\in W^{1,p}(\Sigma)$ of \eqref{minpb}. Moreover, $u$ is a solution of \eqref{case1intro}. 
\item The eigenvalue $\lambda(\beta,\Sigma)$ is simple, that is all the associated eigenfunctions are scalar multiple of each other.
\end{itemize}

\end{prop}
A trivial upper bound for $\lambda(\beta,\Sigma)$ is given, by choosing constant test functions, by
\[
\lambda(\beta,\Sigma)\le \beta \frac{P(D)}{\left|\Sigma\right|}.
\]
Another simple upper bound for $\lambda(\beta,\Sigma)$ is given by
\[
\lambda(\beta,\Sigma) \le \Lambda(\Sigma),
\]
where 
\[
\Lambda(\Sigma)=  \min\left\{ \dfrac{\ds\int _{\Sigma}|\nabla \varphi|^p\;dx}{\ds\int_{\Sigma}|\varphi|^p\;dx},\; \varphi\in W^{1,p}(\Sigma),\; \left.\varphi\right|_{\de D}=0\right\},
\]
that is the first eigenvalue of the corresponding Dirichlet-Neumann problem. 

\begin{prop}
Let $\beta>0$, $D,\Omega$ Lipschitz domains with $D\Subset \Omega$, and $\Sigma=\Omega\setminus \bar D$. The following properties hold.
\begin{enumerate}
\item  The first eigenvalue $\lambda(\beta,\Sigma)$ is differentiable and nondecreasing in $\beta>0$. Moreover,
\[
\frac{d}{d\beta} \lambda(\beta,\Sigma)= \frac{\ds\int_{\de D} |u|^{p}d\sigma}{\ds\int_{\Sigma} |u|^{p}dx}.
\]
where $u$ is a minimizer of \eqref{minpb}.
\item It holds that
\begin{equation}
\label{limiteinf}
\lim _{\beta\to +\infty} \lambda(\beta,\Sigma)=\Lambda(\Sigma).
\end{equation}
\item The first eigenvalue $\lambda(\beta,\Sigma)$ is concave with respect to $\beta$.
\end{enumerate}
\end{prop}
 
\begin{proof}
Let us denote by
\[
Q(\beta,\varphi)=\dfrac{\ds\int_{\Sigma}|\nabla \varphi|^p\;dx+\beta\int_{\de D}  |\varphi |^{p}d\sigma}{\ds\int_{\Sigma}|\varphi|^p\;dx},\qquad \varphi\in W^{1,p}(\Sigma).
\]
Let $u_{\beta}$, $u_{\beta+h}$ be two positive minimizers of \eqref{minpb} associated to $\beta$ and $\beta+h$, respectively. Moreover, suppose $\int_{\Sigma}u_{\beta}^{p}dx=\int_{\Sigma}u_{\beta+h}^{p}dx=1$. Then,
\begin{equation*}
\lambda(\beta+h,\Sigma)-\lambda(\beta,\Sigma)\ge\lambda(\beta+h,\Sigma)- Q(\beta,u_{\beta+h}) = h\int_{\de D}u_{\beta+h}^{p}d\sigma.
\end{equation*}
On the other hand,
\[
\lambda(\beta+h,\Sigma)-\lambda(\beta,\Sigma) \le Q(\beta+h,u_{\beta})-\lambda(\beta,\Sigma) = h\int_{\de D}u_{\beta}^{p}d\sigma.
\]
Finally,
\[
\int_{\de D}u_{\beta+h}^{p}d\sigma\le\frac{\lambda(\beta+h,\Sigma)-\lambda(\beta,\Sigma)}{h}\le 
\int_{\de D}u_{\beta}^{p}d\sigma.
\]
Then, being $\left\{u_{\beta+h}\right\}_{h}$ bounded in $W^{1,p}(\Omega)$, there exists a subsequence, still denoted by $u_{\beta+h}$ such that $u_{\beta+h}\to u_\beta$ strongly in $L^p$ and almost everywhere and $\nabla u_\beta\rightharpoonup \nabla u_\beta$ weakly in $L^p(\Omega)$. As a consequence, by the compactness of the
trace operator (see for example \cite[Cor. 18.4]{L}), $u_{\beta+h}$ converges strongly to $u_\beta$ in $L^p(\de D)$.

The monotonicity of $\lambda(\beta,\Sigma)$ with respect to $\beta$ is obvious.

Now we show the continuity of $\lambda$ with respect to $\beta$. Define
\[
Q(\beta,\varphi)=\dfrac{\ds\int_{\Sigma}|\nabla \varphi|^p\;dx+\beta\int_{\de D}  |\varphi |^{p}d\sigma}{\ds\int_{\Sigma}|\varphi|^p\;dx},\quad \varphi\in W^{1,p}(\Sigma).
\]
Let $\varphi_{k}$ be a minimizing sequence for $\lambda(\beta,\Sigma)$, normalized with $\|\varphi_{k}\|_{L^{p}(\Sigma)}=1$. Then $\int_{\de D}\varphi_{k} d\sigma \le C= C(\Sigma,\beta)$. Hence
\[
\lambda(\beta,\Sigma)\le \lambda(\beta+\eps,\Sigma) \le Q(\beta+\eps,\varphi_{k}) \le Q(\beta,\varphi_{k}) +\eps C.
\]
Passing to the limit as $k\to +\infty$, we get
\[
0\le \lambda(\beta+\eps,\Sigma)-\lambda(\beta,\Sigma)\le \eps C.
\]

Let $u_{\beta}$ be a minimizer for $\lambda(\beta,\Sigma)$, with $\int_{\Sigma}\left|u_{\beta}\right|^{p}dx=1$. Then being
\[
\lambda(\beta,\Sigma)=\int_{\Sigma}\left|\nabla u_{\beta}\right|^{p}dx+\beta \int_{\de D} \left|u_{\beta}\right|^{p}d\sigma \le \lambda(\Sigma),
\]
if $\beta_{j}\to +\infty$ as $j\to +\infty$, there exists $u_{\infty}\in W^{1,p}(\Sigma)$ such that, up to a subsequence,
\[
u_{\beta_{j}}\rightharpoonup u_{\infty} \text{ in }W^{1,p}(\Sigma),\; u_{\beta_{j}}\to  u_{\infty}\equiv 0  \text{ in }L^{p}(\de D),\;u_{\beta_{j}}\to  u_{\infty} \text{ in }L^{p}(\Sigma).
\]
Hence by lower semicontinuity,
\[
\lambda(\Sigma) \le \int_{\Sigma} \left|\nabla u_{\infty}\right|^{p}dx \le \liminf_{j} \lambda(\beta_{j},\Sigma) \le \lambda(\Sigma),
\]
that gives \eqref{limiteinf}.

We finally prove that $\lambda(\beta,\Sigma)$ is concave in $\beta$. Indeed, for fixed $\beta_0\in\mathbb{R}$, we have to show that
	\begin{equation}\label{concavity}
		\lambda(\beta,\Sigma)\leq\lambda(\beta_0,\Sigma)+\frac{d}{d\beta}\lambda(\beta_0,\Omega)\left(\beta-\beta_0\right),
	\end{equation}
	for every $\beta>0$. 
	Let $u_0$ the eigenfunction associated to $\lambda(\beta_0,\Sigma)$ and normalized such that  $\int_{\Sigma}u_0^p\;dx=1$.
Hence, we have 
	\begin{equation}\label{c}
		\lambda(\beta,\Sigma)\le \ds\int_{\Sigma}|D{u_0}|^p\;dx+\beta\int_{\de D}|{u_0}|^p\;d\mathcal{H}^{n-1}.
	\end{equation}
	Now, summing and subtracting to the right hand side of \eqref{c}, $\beta_0\int_{\de D}|u_0|^pd\mathcal{H}^{n-1}$, taking into account that we have just showed that \[ \frac{d}{d\beta}\lambda(\beta_0,\Omega)=\int_{\de D}|{u_0}|^p\;d\mathcal{H}^{n-1},\] we obtain \eqref{concavity}. 

\end{proof}
\paragraph{The radial case}
\begin{prop} %\label{radial_theorem}
Let be $\beta>0$, $R_2> R_1>0$ and let $v$ be a positive minimizer of problem \eqref{minpb} on the spherical shell $\Sigma=A_{R_1,R_2}$. Then $v$ is radially symmetric, in the sense that $v(x)=\psi(|x|)$. Moreover, $\psi'(r)>0$.
\end{prop}

\begin{proof}
	The radial symmetry of the minimizers follows from the rotational invariance of problem \eqref{minpb} and the simplicity of $\lambda(\beta, A_{R_1,R_2})$. 
	
	Now, let $\psi$ be a positive radial eigenfunction. We observe that $\psi$ solves
\begin{equation*}%\label{strongprinc}
\begin{cases}
-\dfrac{1}{r^{n-1}}\left(|\psi'(r)|^{p-2} \psi'(r)r^{n-1}\right)'=\lambda (\beta, A_{R_1,R_2})\,\psi^{p-1}(r)\quad \text{if}\  r\in(R_1,R_2),\vspace{0.2cm}\\
  \psi'(R_2) =0, \vspace{0.2cm}\\
|\psi'(R_1)|^{p-2}\psi'(R_1)+\beta\psi^{p-1}(R_1)=0.
\end{cases}
\end{equation*}
%First of all we consider the case $p=2$; the problem is 
% \begin{equation*}
%-\left(\psi''(r)+(n-1)\dfrac{\psi'(r)}{r} \right)=\lambda \psi(r), 
%\end{equation*}
%that becomes
For every $r\in(R_1,R_2)$ it holds that
\begin{equation*}%\label{sign_derivative}
-\dfrac{1}{r^{n-1}}\left(|\psi'(r)|^{p-2} \psi'(r)r^{n-1}\right)'=\lambda (\beta, A_{R_1,R_2})\psi^{p-1}(r)>0,
\end{equation*}
and, as a consequence,
\begin{equation*}
\left(|\psi'(r)|^{p-2} \psi'(r)r^{n-1}\right)'<0.
\end{equation*}
Taking into account the boundary conditions $\psi'(R_2)=0$, it follows by integrating that $\psi'(r)>0$, and this concludes the proof.
\end{proof}

\subsection{The $p$-torsional rigidity}
Similarly to the case of the first eigenvalue, the following results hold.
\begin{prop}
Let $\beta>0$.  There exists a unique positive maximizer $u_{\Sigma}\in W^{1,p}(\Sigma)$ of \eqref{minpbtor} which solves \eqref{case1introtor}, and any other maximizer of \eqref{case1introtor} is a scalar multiple of $u_{\Sigma}$.
\end{prop}
\begin{prop} %\label{radial_theorem}
Let be $\beta>0$, $r_2> r_1>0$ and let $w$ be a positive maximizer of problem \eqref{minpbtor} on the spherical shell $\Sigma=A_{r_1,r_2}$. Then $w$ is radially symmetric, $w(x)=\Psi(|x|)$, with $\Psi'(r)>0$.
\end{prop}
%A trivial lower bound for $T(\beta,\Sigma)$ is given, by choosing constant test functions, by
%\[
%T(\beta,\Sigma)\ge \beta \frac{P(D)}{\left|\Sigma\right|}.
%\]
%Another simple upper bound for $\lambda(\beta,\Sigma)$ is given by
%\[
%\lambda(\beta,\Sigma) \le \lambda(\Sigma),
%\]
%where 
%\[
%\lambda(\Sigma)=  \min\left\{ \dfrac{\ds\int _{\Sigma}|\nabla \varphi|^2\;dx}{\ds\int_{\Sigma}|\varphi|^2\;dx},\; \varphi\in W^{1,p}(\Sigma),\; \left.\varphi\right|_{\de D}=0\right\},
%\]
%that is the first eigenvalue of the corresponding Dirichlet-Neumann problem. 

\subsection{Quermassintegrals and the Aleksandrov-Fenchel inequalities}
Here we list some basic properties of convex analysis which will be useful in the following. For an extended discussion on the subject we refer the reader to \cite{S}.

Let $K$ be a bounded convex open set, and $B_{1}=\{x\colon \left|x\right|<1\}$. 
The outer parallel body of $K$ at distance $\rho>0$ is the Minkowski sum
\[ 
K+\rho B_1=\{ x+\rho y\in\mathbb{R}^n\;|\; x\in K,\;y\in B_1 \}.
\]
%=\{ x\in\mathbb{R}^n\;|\; d(x)\leq \rho \},$$ 
The Steiner formulas assert that
\begin{equation*}\label{general_steiner}
|K+\rho B_1|=\sum_{i=0}^{n}\binom{n}{i} W_i(\Omega_0)\rho^i,
\end{equation*}
and
\begin{equation*}\label{general_steiner_per}
P(K+\rho B_1)=n\sum_{i=0}^{n-1}\binom{n-1}{i} W_{i+1}(K)\rho^{i}. 
\end{equation*}
The coefficients $W_{i}(K)$ are known as the quermassintegrals of $K$. In particular, it holds that
\[
W_{0}(K)=\left|K\right|,\quad nW_{1}(K)= P(K), \quad W_{n}(K)=\omega_{n}.
\]

The Aleksandrov-Fenchel inequalities state that
\begin{equation}
  \label{afineq}
\left( \frac{W_j(K)}{\omega_n} \right)^{\frac{1}{n-j}} \ge \left(
  \frac{W_i(K)}{\omega_n} \right)^{\frac{1}{n-i}}, \quad 0\le i < j
\le n-1,
\end{equation}
where the inequality is replaced by an equality if and only if $K$ is a ball. 

In what follows, we use the Aleksandrov-Fenchel inequalities for
particular values of $i$ and $j$.
When $i=0$ and $j=1$, we have the classical isoperimetric inequality:
\[
P(K) \ge n \omega_n^{\frac 1 n} |K|^{1-\frac 1 n}.
\]

%In particular, it holds true that
%\begin{equation}\label{AF2}
%W_2(K) \ge n^{-\frac{n-2}{n-1}}{\omega_n}^{\frac{1}{n-1}}P(K)^{\frac{n-2}{n-1}}.
%\end{equation}

Let us denote by $K^{*}$ a ball such that $W_{n-1}(K)=W_{n-1}(K^{*})$. Then by Aleksandrov-Fenchel inequalities \eqref{afineq}, for $0\le i < n-1$
\[
\left(\frac{W_{i}(K)}{\omega_{n}}\right)^{\frac{1}{n-i}}\le
\frac{W_{n-1}(K)}{\omega_{n}} =
\left(\frac{W_{i}(K^{*})}{\omega_{n}}\right)^{\frac{1}{n-i}},
\]
hence 
\begin{equation*}
\label{afapp}
W_{i}(K)\le W_{i}(K^{*}),\quad 0\le i<n-1.
\end{equation*}

\section{Proof of main results}
In this Section we prove the main results (Theorem \ref{herschn} and Theorem \ref{herschntor}). %Firstly we state the main theorem in a precise form.

%\begin{thm}
%\label{herschn}
%	Let $\beta>0$, and $\Omega$ and $D$ be two bounded open sets, with $D$ convex, $\Omega$ Lipschitz, connected, and $D\Subset \Omega$. Let be $\Sigma=\Omega\setminus \bar D$, $A=A_{R_1,R_2}=B_{R_{2}}\setminus \overline{B}_{R_{1}}$, where $B_{R_{i}}$ is a ball centered at the origin with radius $R_{i}$. Suppose that $\left|A\right|=\left|\Sigma\right|$, and $W_{n-1}(B_{R_1})=W_{n-1}(D)$. Then,
%\begin{equation*}
%	\label{in_eig_NR}
%	\lambda(\beta,\Sigma)\le \lambda(\beta,A).
%			% \label{in_tor_NR}			T_p^{NR}(\beta, \Omega)\geq T_p^{NR}(\beta, A).%\\	%\label{in_eig_NR}		\lambda_p^{NR}(\beta, \Omega)\leq \lambda_p^{NR}(\beta, A),\\
%		%	 \label{in_tor_NR}				T_p^{NR}(\beta, \Omega)\geq T_p^{NR}(\beta, A).
%\end{equation*}
%%the annulus $A$ maximizes the first $p$-Laplacian eigenvalue and minimizes the torsion rigidity.
%\end{thm}

\begin{proof}[Proof of Theorem \ref{herschn}]
Let $v(x)=\psi(\left|x\right|)$ be a positive radial solution of problem \eqref{minpb} on $A$, and denote by $v_{m}=\psi(R_{1})$ and $v_{M}=\psi(R_{2})$ be the minimum and maximum of $v$, respectively.

For $x\in \Sigma$, let us denote by $d(x)$ the distance of $x$ from $  D$,
\[
d(x)=\inf\{\left|x-y\right|, y\in D\},
\] 
and consider as test function 
\begin{equation*}\label{utest}
w(x):=
\begin{cases}
G(d(x))\quad\ \ \text{if} \  d(x)< r_2-r_1\\
v_M\qquad\qquad\text{if} \ d(x)\geq r_2-r_1,
\end{cases}
\end{equation*}
where $G$ is defined as 
\begin{equation*}%\label{G_min_t}
	G^{-1}(t)=\int_{v_m}^{t}\dfrac{1}{g(\tau)}\;d\tau,\qquad v_m< t<v_M,
\end{equation*}
 with  $g(t)=|Dv|_{v=t}$. We observe that $v(x)=G(|x|-R_1)$ and $w$ satisfies  the following properties: %$|\nabla u|_{u=t}=|Dv|_{v=t}$ and %$||u||_{\infty}\leq ||v||_{\infty}$ . As before, we define $u_0$ equal to $u$ in $\Omega$ and $u_0$ equal to a suitable constant in $G$. We have:
\begin{gather*}
w\in W^{1,p}(\Sigma)\cap C(\bar\Sigma)\\
|Dw|_{w=t}=|Dv|_{v=t}\\
w_m:=\min_{\bar\Sigma} w=v_m=G(0)\\
w_M:=\max_{\bar\Sigma} w= v_M.
\end{gather*}
Since $R_2-R_1=G^{-1}(v_M)$, to prove that the maximum value of $w$ is $v_M$, we need to verify that $G^{-1}(v_M)=\int_{v_m}^{v_M}\frac 1{g(\tau)}d \tau$. Indeed since $v(x)=\psi(\left|x\right|)$, then $\int_{v_m}^{v_M}\frac 1 {\psi'(v^{-1}(s))}ds=\psi^{-1}(v_M)-\psi^{-1}(v_m)=R_2-R_1$.

Let us denote by  
\[
E_{t}=\overline D\cup \{x\in\Sigma \colon w<t\}, \quad F_{t} =\overline B_{R_{1}}   \cup \{x\in\Sigma \colon v<t\}.
\]
Let us observe that
\[
E_{t} \subseteq \{ x\in \R^{n}\colon d(x) < G^{-1}(t) \}:=\tilde E_{t},\quad F_{t}= \left\{ x\in \R^{n}\colon \left|x\right| < R_{1}+ G^{-1}(t) \right\}.
\]
By Steiner formula and the Aleksandrov-Fenchel inequalities, we get, as $\rho= G^{-1}(t)$, that
\begin{multline}
\mathcal H^{n-1}(\de E_{t}\cap \Sigma)\le P(\tilde E_{t})=
P(D+\rho B_{1}) =n\sum_{k=0}^{n-1}\binom{n-1}{k}W_{k+1}(D)\rho^k\\
 \le n\sum_{k=0}^{n-1}\binom{n-1}{k}W_{k+1}(B_{R_{1}})\rho^k = P(B_{R_{1}}+\rho B)=P(F_{t}).
\label{perimetro_vivo}
\end{multline}
Using now the coarea formula and \eqref{perimetro_vivo}:
\begin{multline}\label{gradient_estimates_e+}
\int_{\Omega}|\nabla w|^p\;dx=%\int_{\Omega\setminus\tilde\Omega}|\nabla u|^p\;dx+\int_{\tilde\Omega}|\nabla u|^p\;dx=
\int_{w_m}^{w_M} g(t)^{p-1}\;\mathcal{H}^{n-1}\left(\partial E_{t}\cap\Sigma \right)dt \\ 
 \leq \int_{w_m}^{w_M} g(t)^{p-1}P(\tilde E_{t})\;dt\le \int_{v_m}^{v_M} g(t)^{p-1}P(F_{t})\;dt=\int_{A}|\nabla v|^p\:dx.
\end{multline}
Since, by construction, $w(x)=w_m=v_m$ on $\de D$, then 
\begin{equation}\label{termine_bordo_me}
\int_{\de D}w^p\:d\mathcal{H}^{n-1}=w^p_m P(D)\leq v^p_m P(B_{r_1})=\int_{\de B_{r_{1}}}v^p\;d\mathcal{H}^{n-1}.
\end{equation}
Now, we define $\mu(t)=|E_{t}\cap\Sigma|$ 
%$\mu(t)=|\tilde{\Omega}|+|E_{0,t}\cap(\Omega\setminus\tilde{\Omega})|$ 
and $\eta(t)=|F_{t}|$ and using again coarea formula, we obtain, for $v_m\le t<v_M$,
\begin{multline*}
%\label{derpri}
\mu'(t)=\int_{\{ w=t\}\cap\Sigma}\dfrac{1}{|\nabla w(x)|}\;d\mathcal{H}^{n-1}=\dfrac{\mathcal{H}^{n-1}\left(  \partial E_{t}\cap\Sigma\right)}{g(t)}\leq \dfrac{P(\tilde E_{t})}{g(t)}\\\leq \dfrac{P(F_{t})}{g(t)}=\int_{\{ v=t\}}\dfrac{1}{|\nabla v(x)|}\;d\mathcal{H}^{n-1}=\eta'(t).
\end{multline*}
This inequality is trivially true also if $0<t<v_m$.
Since $\mu(0)=\eta(0)=0$, by integrating from $0$ to $t<v_{M}$, we have:
 \begin{equation*}\label{magmu}
\mu(t)\leq\eta(t).
\end{equation*}
On the other hand, we have
\begin{multline}\label{L_p_estimates_e+}
	\int_{\Sigma}w^p\;dx=\int_{0}^{v_M}pt^{p-1}(|\tilde\Omega| -\mu(t))dt\ge \int_{0}^{v_M}pt^{p-1}(|A|-\eta(t))\;dt= \int_{A}v^p \;dx.
\end{multline}
Using \eqref{gradient_estimates_e+}, \eqref{termine_bordo_me} and \eqref{L_p_estimates_e+}, we achieve
\begin{equation*}
\begin{split}
\lambda(\beta, \Sigma)&\leq \dfrac{\ds\int_{\Sigma}|\nabla w|^p\;dx+\beta\ds\int_{\de D}w^p\;d\mathcal{H}^{n-1}    }{\ds\int_{\Sigma}w^p\;dx} \\ 
& \leq \dfrac{\ds\int_{A}|\nabla v|^p\;dx+\beta\ds\int_{\de B_{r_1}}v^p\;d\mathcal{H}^{n-1}    }{\ds\int_{A}v^p\;dx}=\lambda (\beta, A),
\end{split}
\end{equation*}
and this conclude the proof.
\end{proof}

\begin{proof}[Proof of Theorem \ref{herschntor}]
As regards the case of the torsional rigidity, the proof follows line by line the proof of Theorem \ref{herschn} making use of the stress function $v_{A}$ instead of the first eigenfunction on the spherical shell.
\end{proof}

\begin{rem}
We stress that with the same argument it is possible to prove that if
\begin{equation*}
\label{minpbgen}
\lambda_{q}(\beta,\Sigma)=  \min\left\{ \dfrac{\ds\int _{\Sigma}|\nabla \varphi|^p\;dx+\beta\int_{\de D}  |\varphi |^{p}d\sigma}{\left(\ds\int_{\Sigma}|\varphi|^q\;dx\right)^{\frac{p}{q}}},\; \varphi\in W^{1,p}(\Sigma),\;\varphi\not\equiv 0\right\},
\end{equation*}
with $1\le q\le p$ and $1<p<+\infty$, it holds that
\[
\lambda_{q}(\beta,\Sigma)\le \lambda_{q}(\beta,A)
\]
where $A=B_{R_{2}}\setminus \bar B_{R_{1}}$ is the spherical shell with $\left|A\right|=\left|\Sigma\right|$ and $W_{n-1}(B_{R_{1}})=W_{n-1}(D)$.
\end{rem}

\section*{Acnowledgements}
This work has been partially supported by a MIUR-PRIN 2017 grant ``Qualitative and quantitative aspects of nonlinear PDE's'' and by GNAMPA of INdAM. The second author (G.P.) was also supported by Progetto di eccellenza \lq\lq Sistemi distribuiti intelligenti\rq\rq of Dipartimento di Ingegneria Elettrica e dell'Informazione \lq\lq M. Scarano\rq\rq.

\end{document}